\documentclass[10pt,twoside]{amsart}
\usepackage{amsmath, amsfonts, amsthm, amssymb} 
\newtheorem{definition}{Definition}[section]
\newtheorem{lemma}[definition]{Lemma}
\newtheorem{theorem}[definition]{Theorem}
\newtheorem{proposition}[definition]{Proposition}

\newtheorem{remark}[definition]{Remark}

\numberwithin{equation}{section} 
\newcommand \Phit {\widetilde \Phi} 
\newcommand \xih {\widehat \xi} 
\newcommand \Hcal {\mathcal H} 
\newcommand \rhob {\overline \rho}
\newcommand \ub {\overline u}
\newcommand \ut {\widetilde u}

\newcommand \Ut {\widetilde U}
\newcommand \Zt {\widetilde Z}

\newcommand \be   {\begin{equation}}
\newcommand \ee   {\end{equation}} 
\newcommand \R      	{\mathbb{R}} 
\newcommand \Rn      	{\mathbb{R}^n} 
\newcommand \eps      \epsilon  
\newcommand \del \partial
\newcommand \tr  {\text{tr}}
\newcommand \supp  {\text{supp}}

\newcommand \para \parallel 
\newcommand \deltab {\overline\delta}

\begin{document} 
\bibliographystyle{plain}  
\title[A symmetrization of the relativistic Euler equations]
{A symmetrization of the relativistic Euler equations in several spatial variables} 
\author[P.G. L{\tiny e}Floch and S. Ukai]
{
%
Philippe G. L{\scriptsize e}Floch and Seiji Ukai  
}  
\thanks{\noindent 
$^1$ Laboratoire Jacques-Louis Lions 
\& Centre National de la Recherche Scientifique, 
Ê      Universit\'e Pierre et Marie Curie (Paris 6), 4 Place Jussieu, 75252 Paris, 
 Ê     France. 
 \\
 Email : {\tt pgLeFloch@gmail.com} 
 \newline 
$^2$  Department of Mathematics, City University of Hong Kong, 
 83 Tat Chee Avenue, Kowloon Tong, Hong Kong. Email : {\tt suk03919@ec.catv-yokohama.ne.jp.}
\, 
\textit{\ 2000 AMS Subject Classification:} 35L65, 76N10.
\textit{Key Words:} relativistic fluid, Euler equations, vacuum, symmetric hyperbolic system.
}  

\begin{abstract}   
We consider the Euler equations governing    
relativistic compressible fluids evolving in the Minkowski spacetime with several 
spatial variables.  
We propose a new symmetrization 
which makes sense for solutions containing vacuum states 
and, for instance, applies to the case of compactly supported solutions, which are
important to model star dynamics.
Then, relying on these symmetrization and assuming that the velocity does not exceed some threshold
and remains bounded away from the light speed, 
we deduce a local-in-time existence result for solutions containing vacuum states. 
We also observe that the support of compactly supported solutions does not expand as time evolves. 
\end{abstract}

\date{December 8, 2008}
\maketitle

\tableofcontents

\newpage 


\section{Introduction and main result}
\label{IN-0}

The dynamics of relativistic compressible fluids evolving in  the
Minkowski spacetime with $n$ spatial variables 
is governed by the Euler equations (for instance \cite{Landau})
\be
\label{Euler}
\aligned 
&\del_t\Bigl({\rho + \eps^2 \, p \over 1 - \eps^2 \, |u|^2} - \eps^2 \, p \Bigr)
  + \sum_{k=1}^n \del_{x_k} \Bigl( {\rho + \eps^2 p \over 1 - \eps^2 |u|^2} \, u_k \Bigr) = 0,
\\
& \del_t \Bigl( {\rho + \eps^2 p \over 1 - \eps^2 |u|^2} \, u_j \Bigr) 
  + \sum_{k=1}^n \del_{x_k} \Bigl( {\rho + \eps^2 p \over 1 - \eps^2 |u|^2} \, u_j u_k + p \, \delta_{jk} \Bigr) = 0. 
\endaligned  
\ee
Here, $\rho$ and $u=(u_j)_{1 \leq j \leq n}$ denote the mass density and the 
($n$-dimensional) velocity vector of the fluid
and functions of the variables $(t,x)\in \R_+ \times\R^n$, while the parameter $1/\eps$ represents the light speed
and $\delta_{jk}$ denotes the Kronecker symbol. 
The range of physical interest for the unknown $(\rho,u)$ is defined by 
\be
\label{constraints}
\rho \geq 0, \qquad |u|^2 := \sum_{j=1}^n u_j^2 < \eps^{-2}, 
\ee 
while the pressure $p=p(\rho)$ is assumed to satisfy 
\be
\label{pressure}
0 \leq p'(\rho) < \eps^{-2}.  
\ee 
Under these conditions, it can be checked that the system of conservation laws \eqref{Euler}
is symmetric hyperbolic as long as vacuum is avoided, i.e.~under the restriction $\rho>0$. 
That is, it can be written in the symmetric hyperbolic in the so-called entropy variables
(Makino and Ukai \cite{MakinoUkai1,MakinoUkai2})
$$
\del_t W+ \sum_{k=1}^n A_k(W) \, \del_{x_k} W = 0,  
$$ 
where $W := (\rho,u) \in \R_+ \times B_{1/\eps}$ 
(the set $B_{1/\eps}$ being the open ball with radius $1/\eps$)
and such that, for every unit vector $\nu= (\nu_k) \in \Rn$, 
the matrix $\sum_{k=1}^n \nu_k A_k(W)$ admits real eigenvalues and a basis of eigenvectors. 
It is also established in \cite{MakinoUkai1,MakinoUkai2}.
the initial value problem with non-vacuum initial data, i.e. 
\be
\label{data}
(\rho, u)(0, \cdot) = (\rhob, \ub),  
\ee
when the initial data $(\rhob, \ub) \in \R_+ \times B_{1/\eps}$ is bounded away from vacuum, 
admits a local-in-time solution.

In the present paper, we are interested in the symmetrization and 
the local-in-time existence for the relativistic fluid equations \eqref{Euler}
when the initial data $(\rhob, \ub)$ take arbitrary values in $\R_+ \times B_{1/\eps}$ 
and are allowed to contain vacuum states.  
Some partial but pioneering results were obtained on this problem by Rendall \cite{Rendall}
and Guo and Tahvildar-Zadeh \cite{GTZ}.  
For an overview of the standard theory, we refer to the relevant the chapter on the Euler equations 
in Choquet-Bruhat's book \cite{ChoquetBruhat}. 
 
We emphasize that compactly supported solutions are important in the applications, for instance 
to model the dynamics of stars. 
However, the current existence theory does not cover this situation; indeed,  
the transformation proposed in \cite{MakinoUkai1,MakinoUkai2} 
does not apply for our purpose since the coefficients $A_k(W)$ therein blow-up near the vacuum. 
On the other hand, 
when $\eps=0$, the system \eqref{Euler} reduces to the non-relativistic Euler equations, for which 
local existence of solutions, even solutions containing vacuum states, was established by 
Makino, Ukai, and Kawashima \cite{MakinoUkaiKawashima}. The objective of the present paper is 
precisely 
to provide a suitable generalization of the theory in \cite{MakinoUkaiKawashima} to 
encompass relativistic fluids. 
This will be achieved by introducing yet another symmetrization 
which significantly differs from previously proposed ones. 
  
An outline of this paper is as follows. In Section~\ref{sym}, we begin our investigation of the Euler equations
and derive a first version of our symmetrization, which is based on using as main unknowns 
``generalized Riemann invariants'' and a ``normalized velocity''. 
Then, in Section~\ref{Lore}, we recall basic material on Lorentz transformations and 
establish a technical lemma. 
Next, in Section~\ref{exist} we are in a position to establish the existence result of this paper.
The main technical difficulty is to check a positive-definiteness property for the symmetric system.  
Finally, in Section~\ref{general} we conclude with some remark about the support of solutions.


\section{Symmetrization of the relativistic Euler equations}
\label{sym}

\subsection*{Modified mass and velocity variables}

We are going to define new variables defined by nonlinear transformations of the mass density and the norm of 
the velocity vector, which allow us to put the relativistic Euler equations
in a symmetric form. To begin with, we introduce the {\sl modified mass density} variable $w$ by 
\be
\label{var-w}
\aligned
& w = w(\rho) : = \int_0^\rho {c(s) \over q(s) } \, ds, 
\\
& c(\rho):= \sqrt{p'(\rho)}, \qquad q(\rho) := \rho + \eps^2 \, p(\rho),
\endaligned
\ee
where $c(\rho):= \sqrt{p'(\rho)}$ represents the sound speed in the fluid.  

From now on, we assume that $w(\rho)$ defined above is {\sl finite.}   
This is the case if, near the vacuum, the equation of state is 
asymptotic to the one of polytropic perfect fluids, i.e. $p(\rho) \sim k \, \rho^\gamma$ 
with $\gamma >1$ and $k>0$. In the special case $p(\rho) = k \, \rho^\gamma$
and when $\eps$ is taken to tend to $0$, then the function $w(\rho)$ approaches 
$\rho^{(\gamma-1)/2}$ (up to a multiplicative constant), which precisely coincides with the function 
introduced in \cite{MakinoUkaiKawashima} in the non-relativistic case.

In addition, based on the norm $|u|$ of the velocity vector $u$, we define the {\sl modified velocity scalar} 
\be
\label{var-v}
v = v(|u|) := {1 \over 2 \eps} \, \ln\Big( {1 + \eps \, |u| \over 1 - \eps \, |u|} \Big). 
\ee
We also introduce the $n$-dimensional, {\sl normalized velocity vector} and the 
associated {\sl projection operator} 
\be
\label{var-ut} 
\ut := {u \over |u|}, \qquad E(u) := I - \ut \otimes \ut,
\ee 
respectively, where $I$ denotes the $n\times n$ identity matrix. 
Observe that $E(u)$ is singular as a function of $u$, when $u$ is close to origin;
however, the map $|u|^2 \, E(u)$ is actually smooth.

In the rest of this section, $(\rho,u)$ denotes a given smooth solution to \eqref{Euler}.

\begin{proposition} [Formulation in terms of the modified mass and velocity variables]
\label{prop1}
The relativistic Euler equations are equivalent to the following system in the variables 
$(w, v, \ut)$: 
$$ 
\aligned 
 \big(1 - \eps^4 |u|^2 c(\rho)^2 \big) \, \del_t w 
   &+ \big(1 - \eps^2 \, c(\rho)^2 \big) \, u \cdot \nabla w
   \\&
+ c(\rho) \, (1 - \eps^2 |u|^2) \, \ut \cdot \nabla v 
+ c(\rho) \, |u| \nabla \cdot \ut = 0,  
\\
 \big(1 - \eps^4 |u|^2 c(\rho)^2 \big) \, \del_t v   &+ c(\rho) \, (1 - \eps^2 |u|^2) \, \ut \cdot \nabla w 
 \\&
    + \big(1 - \eps^2 \, c(\rho)^2 \big) \, u \cdot \nabla v
-\eps^2 c(\rho)^2 \, |u|^2 \nabla \cdot \ut = 0,  
\\
 (1 - \eps^2 |u|^2)^{-1} \, |u| \, \big( \del_t &\ut + u\cdot \nabla \ut \big) 
    + c(\rho) \, E(u) \, \nabla w =0.  
\endaligned 
$$
\end{proposition} 

\begin{proof} 
\noindent{\it Step 1.} The first equation in \eqref{Euler} takes the form 
$$
\aligned 
& {q'(\rho) \over 1 - \eps^2 \, |u|^2}\, \del_t \rho 
+ q(\rho) {2 \eps^2 \over (1 - \eps^2 \, |u|^2)^2} \, u \cdot \del_t u
       - \eps^2 \, p'(\rho) \, \del_t \rho
\\
& + \sum_{k=1}^n {q'(\rho) \over 1 - \eps^2 \, |u|^2}  \, u_k \del_{x_k} \rho 
    + q(\rho) {2 \eps^2 \over (1 - \eps^2 \, |u|^2)^2} \, u_k \, u \cdot \del_{x_k} u
    + {q(\rho) \over 1 - \eps^2 |u|^2} \, \del_{x_k} u_k = 0,
\endaligned 
$$
or equivalently 
$$
\aligned 
& {1 + \eps^2 |u|^2 (q'(\rho)-1) \over 1 - \eps^2 \, |u|^2} \, \del_t \rho 
   + {q'(\rho) \over 1 - \eps^2 \, |u|^2} \, u \cdot \nabla \rho 
\\
& + {\eps^2 \,  q(\rho) \over (1 - \eps^2 \, |u|^2)^2} \, \big( \del_t |u|^2 + u \cdot \nabla |u|^2 \big)   
    + {q(\rho)\over 1 - \eps^2 |u|^2} \, \nabla \cdot u = 0. 
\endaligned 
$$  
By multiplying this equation by $w'(\rho)$ we find 
$$ 
\aligned 
& {1 + \eps^4 |u|^2 c(\rho)^2 \over 1 - \eps^2 \, |u|^2} \, \del_t w 
   + {1 + \eps^2 \, c(\rho)^2 \over 1 - \eps^2 \, |u|^2} \, u \cdot \nabla w
\\
& + {\eps^2 \, c(\rho) \over (1 - \eps^2 \, |u|^2)^2} \, \big( \del_t |u|^2 + u \cdot \nabla |u|^2 \big)   
    + {c(\rho) \over 1 - \eps^2 |u|^2} \, \nabla \cdot u = 0. 
\endaligned 
$$

To rewrite the above equation in a more convenient form, we observe that
$$
{dv \over d|u|} = {1 \over 1 - \eps^2 \, |u|^2}, 
$$
and so, after further multiplication by $(1 - \eps^2 \, |u|^2)$, the equation for 
the modified mass density reads 
\be
\label{two2} 
\aligned 
\big(1 +\eps^4 |u|^2 c(\rho)^2 \big) \, \del_t w 
   & + \big(1 + \eps^2 \, c(\rho)^2 \big) \, u \cdot \nabla w
\\
& + 2 \eps^2 c(\rho) \, |u| \big( \del_t v + u \cdot \nabla v\big)
    + c(\rho) \, \nabla \cdot u = 0. 
\endaligned 
\ee

\

\noindent{\it Step 2.} Next, we expand the second equation in \eqref{Euler} and obtain 
$$
\aligned 
\del_t \big({q(\rho) \over 1 - \eps^2 \, |u|^2}\big) \, u 
& + 
{q(\rho) \over 1 - \eps^2 \, |u|^2} \, \del_t u 
\\
& + u \, \nabla \cdot \big({q(\rho) \over 1 - \eps^2 \, |u|^2} \, u \big) 
  + {q(\rho) \over 1 - \eps^2 \, |u|^2} \, u\cdot \nabla u 
  + p'(\rho) \nabla \rho = 0, 
\endaligned
$$
which, after using the mass equation, yields 
$$
\aligned 
{q(\rho) \over 1 - \eps^2 \, |u|^2}\, \big( \del_t u + u\cdot \nabla u \big) 
+ p'(\rho) \, \big( \eps^2 \del_t \rho \, u + \nabla \rho \big) = 0. 
\endaligned
$$
Multiplying by $1/q(\rho)$ we arrive at an equation for the velocity vector 
\be
\label{three} 
\aligned 
(1 - \eps^2 \, |u|^2)^{-1} \, \big( \del_t u + u\cdot \nabla u \big) 
+ c(\rho) \, \big( \eps^2 \del_t w \, u + \nabla w \big) = 0. 
\endaligned
\ee  

We now multiply \eqref{three} by the vector $u$ itself, and obtain 
$$
\aligned 
(1 - \eps^2 \, |u|^2)^{-1} \, \big( \del_t |u|^2 + u \cdot \nabla |u|^2 \big) 
    + 2 c(\rho) \, \big( \eps^2 \del_t w \, |u|^2 + u \cdot \nabla w \big) = 0,
\endaligned
$$ 
which, after a further multiplication by $(2 |u|)^{-1}$, becomes
\be
\label{three2} 
\aligned 
\del_t v + u \cdot \nabla v
    + c(\rho) \, \big( \eps^2 |u| \del_t w + \ut \cdot \nabla w \big) = 0. 
\endaligned
\ee  
\

\noindent{\it Step 3.} 
To derive the equation for $\ut$ we multiply \eqref{three} by the projection matrix $E(u)$ and obtain 
$$
(1 - \eps^2 |u|^2)^{-1} \, E(u) \big( \del_t u + u\cdot \nabla u \big) 
+ c(\rho) E(u) \big( \eps^2 \, \del_t w \, u + \nabla w \big) = 0. 
$$
In view of the identities 
$$
\aligned 
E(u) u & = 0,
\\
E (u) \, \del_t u & = \del_t u - \ut \, \del_t |u| = |u| \del_t \ut, 
\\
E(u) \big( u\cdot \nabla u\big) & = |u| u\cdot \nabla \ut, 
\endaligned 
$$
we arrive at the (third) equation for the normalized velocity $\ut$, as stated in the proposition.

Finally, we are in a position to return to the equation \eqref{two2} and,  
by plugging \eqref{three2} in \eqref{two2}, find
$$
\big(1 - \eps^4 |u|^2 c(\rho)^2 \big) \, \del_t w 
   + \big(1 - \eps^2 \, c(\rho)^2 \big) \, u \cdot \nabla w + c(\rho) \, \nabla \cdot u = 0.  
$$ 
Hence, observing that  
$$
\nabla \cdot u = (1 - \eps^2 |u|^2) \, \ut \cdot \nabla v + |u| \nabla \cdot \ut,
$$
we obtain the desired equation for $w$, as stated in the proposition. In turn, we can also put \eqref{three} in the form stated in the proposition
for the function $v$. 
\end{proof}


\subsection*{A symmetric hyperbolic formulation}

At this juncture, it may be interesting to consider the one-dimensional case $n=1$. 
Considering the result in Proposition~\ref{prop1} and setting 
$u_x:=\nabla u$, etc. and then 
observing that (in the one-dimensional case) $E \equiv 0$ and $\ut \equiv 1$,
we obtain the following form of the relativistic Euler equations in the variables $(w,v)$ 
$$
\aligned 
& \big(1 - \eps^4 u^2 c(\rho)^2 \big) \, \del_t w 
   + \big(1 - \eps^2 \, c(\rho)^2 \big) \, u w_x + c(\rho) \, (1 - \eps^2 u^2) \, v_x = 0,  
\\
& \big(1 - \eps^4 u^2 c(\rho)^2 \big) \, \del_t v 
   + c(\rho) \, (1 - \eps^2 u^2) \, w_x   + \big(1 - \eps^2 \, c(\rho)^2 \big) \, u \, v_x = 0,  
\endaligned 
$$
which, obviously, is a symmetric hyperbolic system. In particular, it is obvious that the coefficient 
$1 - \eps^4 u^2 c(\rho)^2$ remains bounded away from zero, provided the sound speed or 
the fluid velocity scalar (or both) remain bounded away from the light speed. 

To derive our symmetrization in general dimension, 
we need two additional observations.
\begin{itemize}
\item First, since $\ut$ has unit norm we can write 
$$
\nabla \cdot \ut 
= \nabla \cdot \ut - \ut \cdot \nabla {|\ut|^2 \over 2} 
= \tr \big( E(u) \, \nabla \ut\big), 
$$
where ``$\tr$" denotes the trace of a matrix. This allows us to rewrite the last term of the $w$-equation
 in Proposition~\ref{prop1}, in the form
$$
c(\rho) \, |u| \nabla \cdot \ut = c(\rho) \, |u| \tr (E(u) \, \nabla \ut). 
$$
Interestingly enough, this term can now be viewed as the ``symmetric counterpart'' of the term 
$$
|u| \, c(\rho) \, E(u) \, \nabla w
$$ 
already contained in the $\ut$-equation in Proposition~\ref{prop1}.

\item At this stage, only one term poses some problem if we are to reach the desired symmetric form, 
that is, the term 
$$
-\eps^2 c(\rho)^2 \, |u|^2 \nabla \cdot \ut
$$
in the $v$-equation of Proposition~\ref{prop1}. 
To compensate for this term, one would need to have 
the term $-\eps^2 c(\rho)^2 \, |u|^2 \, \nabla v$ in the $\ut$-equation, but 
it does not appear that a direct transformation could achieve this. 
So, we introduce a further transformation based based on still some new unknowns:  
\be
\label{var-z}
z_\pm := v \pm w,
\ee
which we will refer to as the {\sl generalized Riemann invariant variables.} 
According to Proposition~\ref{prop1}, we have the equations 
\begin{align*}
& (1 - \eps^4 |u|^2 c(\rho)^2 ) \, \del_t z_\pm 
    + (1 - \eps^2 \, c(\rho)^2) (|u| \pm c(\rho)) \, \ut \cdot \nabla z_\pm
\\ 
& 
\hskip5.cm 
\pm  (1 \mp \eps^2 c(\rho) |u|) \, c(\rho) \, |u| \, \nabla \cdot \ut = 0,  
\\
&(1 - \eps^2 |u|^2)^{-1} \, |u|^2 \big( \del_t \ut + u\cdot \nabla \ut \big)
     + {1 \over 2} |u| \, c(\rho) \, E(u) \, \big( \nabla z_+ - \nabla z_- \big) = 0. 
\end{align*}
 
\end{itemize}

\

Consequently, by combining together the above observations we arrive at the main conclusion of this section: 

\begin{proposition}[Symmetric form of the Euler equations]
\label{FORM1}
In terms of the generalized Riemann invariant variables $(z_+, z_-)$ and the normalized velocity $\ut$
defined in \eqref{var-w}, \eqref{var-v}, \eqref{var-z},  
the relativistic Euler equations take the following symmetric form  
\be
\label{simple-Zp2}
\aligned 
& (1 +\eps^2 |u| c(\rho)) \, \del_t z_+ 
   + {1 - \eps^2 \, c(\rho)^2 \over 1 - \eps^2 c(\rho) |u|} \, (|u|+c(\rho)) \, \ut \cdot \nabla z_+
\\
& \hskip5.cm  + c(\rho) \, |u| \, \tr (E(\ut) \, \nabla \ut) = 0,  
\endaligned 
\ee \be
\label{simple-Zm2}
\aligned 
& (1 - \eps^2 |u| c(\rho)) \, \del_t z_- 
   + {1 - \eps^2 \, c(\rho)^2 \over 1 + \eps^2 c(\rho) |u|} \, (|u| - c(\rho)) \, \ut \cdot \nabla z_-
\\
& \hskip5.cm    - c(\rho) \, |u| \, \tr (E(\ut) \, \nabla \ut) = 0,  
\endaligned 
\ee 
\be
\label{simple-ut3}
{2 \, |u|^2 \over 1 - \eps^2 |u|^2} \, \big( \del_t \ut + u\cdot \nabla \ut \big) 
      + c(\rho) \, |u| \, E(\ut) \, \nabla z_+ - c(\rho) \, |u| \, E(\ut) \, \nabla z_- = 0,
\ee
where the unknowns $z_\pm$ are real-valued and $\ut$ is a unit vector, $|\ut|=1$. 
\end{proposition}

Here, the quantity $\rho$ must be regarded as a function of the variable $z_+ - z_- = 2w$, 
where $w$ was defined earlier as a function of $\rho$. On the other hand, $u$ is a 
function of $z_+ - z_-$, namely 
$$
\eps |u| = {e^{\eps (z_+ + z_-)} - 1 \over e^{\eps (z_+ + z_-)} + 1} 
$$

Observe that the above symmetric formulation does allow the density variable to vanish, since the coefficients 
above remain bounded as the density approaches the vacuum. 
However, since the coefficient in front of $\del_t \ut$ in the third equation 
vanishes with $u$, we see that the above formulation requires the velocity 
$u$ to be bounded away from the origin which, of course, is not a realistic assumption to put on
general solutions with vacuum. In Section 4, however, 
we will discuss a reduction of the general initial value problem
which ensures this condition  
after applying a well-chosen Lorentz transformation to an arbitrary solution.

\begin{remark} 
\label{rm2.3} 
The assumed normalization $|\ut|^2 = 1$ could be relaxed in the formulation of the system. 
In fact, if this condition holds at the initial time, then it holds for all times, 
as is clear from the transport equation satisfied by $|\ut|$
$$
\del_t |\ut| + u\cdot \nabla |\ut| = 0,  
$$
which follows from \eqref{simple-ut3} multiplied by $\ut$. 
\end{remark}


\section{Properties of Lorentz transformation}
\label{Lore} 

\subsection*{Transformation formulas}

We will need to rely on the Lorentz invariance property of the relativistic Euler equations and, therefore, in the present section, we collect several technical results about Lorentz transfomations.

For every $U \in \Rn$ with $U \neq 0$, we set $\Ut := U/|U|$ 
and we decompose any vector $x \in \Rn$ in a unique way such that 
$$
x =  x_\para \, \Ut 
+ x_\perp, \qquad x_\para = x \cdot \Ut \in \R, 
\qquad x_\perp \cdot U = 0. 
$$
The {\sl Lorentz transformation} $(t,x) \mapsto (t',x')$ associated with the vector $U$ is then defined by 
$$
\aligned 
t' & = \gamma(U) \, ( t - \eps^2 U \cdot x), 
\\
x_\para' & = \gamma(U) \, (x_\para - U_\para \, t), 
\\
x_\perp' & = x_\perp,
\endaligned 
$$
where 
\[
\gamma(U) =\frac{1}{\sqrt{1 - \eps^2 \, |U|^2}}
\]
 is the so-called Lorentz factor. 
This transformation can be put in an equivalent form 
\be
\label{Lorentz} 
\aligned 
t' & = \gamma(U) \, (t - \eps^2 \, U \cdot x), 
\\
x' & = -\gamma(U) U t +\big(I + (\gamma(U)-1)\Ut \otimes \Ut \big) x. 
\endaligned 
\ee

It may be also convenient to use the modified velocity scalar $V$ 
associated with $U$ (following the definition in the previous section)
and given by 
$$
e^{\eps V} := \gamma(U) \, (1 + \eps \, |U|) = \Big( {1 + \eps |U| \over 1 - \eps |U|} \Big)^{1/2},
$$
and to rewrite the Lorentz transformation as 
$$
\aligned 
& (t' \pm \eps \, \Ut \cdot x') = e^{\mp \eps V} \, (t \pm \eps \, \Ut \cdot x), 
\\
& x_\perp' = x_\perp. 
\endaligned 
$$
or, equivalently, as 
\be
\label{Lorentz2}
\aligned 
t' & = \cosh(\eps V) \, t - \eps \, \sinh(\eps V) \, \Ut \cdot x,   
\\
x' & = - \sinh(\eps V) \, t + \eps \, \cosh(\eps V) \, \Ut \cdot x,
\\
x_\perp' & = x_\perp.
\endaligned
\ee

Recall also that Lorentz transformations together with spatial rotations 
form the so-called {\sl Poincar\'e group of isometries,} 
characterized by the condition that the length element of the Minkowski metric is 
preserved, that is,   
$$
- \eps^{-2} \, {t'}^2 + x'_\para{^2} + |x_\perp'|^2 
= - \eps^{-2} \, t^2 +  x_\para^2 + |x_\perp|^2. 
$$

Recall also that the relativistic Euler equations are invariant under Lorentz transformations
which, for instance, can be checked by direct (tedious) calculations or from more abstract considerations. 
The following transformation rule will also be useful in the following section. 
 
\begin{lemma}[Velocity transformation formula] 
\label{additionlaw} 
Let  $U \in \Rn$ with $U \neq 0$, and denote by $u, u'$ the fluid velocity vectors in 
different coordinate systems $(t,x), (t',x')$ related by the Lorentz transformation \eqref{Lorentz}.  
Then, the transformation law for these velocity vectors is 
\be
\label{addition}
\aligned 
u' & = \frac{1}{1-\eps^2 U\cdot u}
      \Big(-U + \Big( \gamma(U)^{-1} I + (1-\gamma(U)^{-1}) \, \Ut\otimes \Ut \Big) \, u \Big)
   \\
   & =: {1 \over \eps} \Phi(\eps u, \eps U).
\endaligned
\ee
\end{lemma}

\begin{proof} The fluid velocity vector represents the velocity
of a fictitious point-mass moving along with the fluid. Thus, $u=u(t,x)$ 
represents the velocity vector of a point-mass located at $x$ at the time $t$, 
while $u'=u'(t',x')$ is the velocity of the {\sl same} point-mass
in the coordinate system $(t',x')$. Consequently, the vectors $u$ and $u'$ are given by
\be
\label{uuprime}
u := {dx \over dt}, \qquad u' := {dx' \over dt'}. 
\ee
Now, in view of \eqref{Lorentz},
\be
\aligned
u' & = {dx' \over dt} \, \Big({dt' \over dt} \Big)^{-1}
\\
& = \Big(- \gamma(U)U 
        + \big(I+(\gamma(U) - 1) \Ut \otimes \Ut \big) \, {dx \over dt} \Big)
         \Big( \gamma(U) \, \big( 1 - \eps^2 \, U \cdot {dx \over dt} \big) \Big)^{-1},
\endaligned
\ee 
which, together with \eqref{uuprime}, yields \eqref{addition}. 
\end{proof}


\subsection*{A technical property on the Lorentz-transformed velocity} 

In the following, we will need to have a lower bound on the fluid velocity vector, so we establish here
a preliminary estimate. Throughout, $\eps \in (0,1)$ and all of the constants are {\sl independent} 
of $\eps$. Given $r_0 \in (0,1)$, we define 
$$
B_{r_0} : = \{ \eps u \in\R^3 \ | \ \eps |u| \leq r_0 \}. 
$$

\begin{lemma}[Uniform bounds for the velocity]
\label{lowerbound}
Given any $r_0 \in (0,1)$ and any vector $U \in \R^3$ satisfying $r_0 < \eps |U| < 1$, 
there exist positive constants $0 < \delta_1 < \delta_2 < 1$ depending only on
$r_0$ and $\eps U$, such that the Lorentz transformed velocity \eqref{addition} has a norm uniformly 
bounded away from, both, the origin and the light speed, i.e. 
$$
\delta_1 \le |\Phi(\eps u, \eps U)| \leq \delta_2 
$$
hold for any $\eps u \in B_{r_0}$. 
\end{lemma}

We observe that the above statement is sharp, in the sense that 
the constants $\delta_1, \delta_2$ may approach the endpoints of the interval $(0,1)$ when 
$r_0$ also approaches the endpoints of $(0,1)$ or when 
$U$ approaches the endpoints of the interval $(r_0,1)$.

\begin{proof} To simplify the notation, we introduce the new variables and function 
$$
X := \eps \, u, \quad Z := \eps \, U, \quad W(X,Z) := |\eps \, u'|^2 = |\Phi(\eps u, \eps U)|^2.  
$$
It is convenient to choose the coordinate system so that
$$
Z = (r_1,0,0), \qquad 0<r_0<r_1<1.
$$

Setting $X=(X_1,X_2,X_3)$ and noting that $\Zt\otimes \Zt X=(X_1,0,0)$, we get
\begin{align}\label{W}
W(X,Z)
& = \Big|\frac{1}{1-X\cdot Z} \Big(\gamma(U)^{-1} X+(1-\gamma(U)^{-1})\Zt\otimes \Zt X - Z\Big)\Big|^2
\\ \notag
& = {1 \over (1-r_1X_1)^2} \, \Big( (X_1-r_1)^2+\gamma(U)^{-2} \, (X_2^2+X_3^2)\Big).
\end{align}
The norm $|Z|=r_1$ being fixed, we are going now to compute the extremum values of the function 
$W$ within the domain $|X| \leq r_0$.

We first compute the minimum by noting that
$$
W(X,Z) \geq \Big( {X_1-r_1 \over 1 - r_1X_1} \Big)^2 =: g(X_1).
$$
An obvious lower bound is $W(X,Z) \geq (r_1 - r_0)^2 (1 + r_1 r_0)^{-2}$. 
To obtain an optimal bound, we compute 
$$
\aligned 
g'(X_1)  
& = 2 \, \frac{X_1-r_1}{1-r_1X_1} \frac{1-r_1^2}{ (1-r_1X_1)^2},
\endaligned
$$
which is negative for any $|X_1|\le r_0<r_1$. Consequently, the minimum
is attained at $X_1=r_0$, and
\be
\aligned
\eps |u'| & = W(X,Z)^{1/2} \geq  g(r_0)^{1/2}
\\
& = {r_1-r_0 \over 1-r_0r_1} =: \delta_1>0.
\endaligned
\ee

Next, to compute the maximum, we set $|X| = r \leq r_0$. Noting that
$\gamma(U)^{-2} = 1 - |Z|^2 = 1 - r_1^2$, we get
$$
\aligned
(&X_1 -r_1)^2+ \gamma(U)^{-2} (X_2^2+X_3^2)
\\
& = (1-X_1r_1)^2+X_1^2-1+r_1^2-r_1^2X_1^2+(1-r_1^2)(X_2^2+X_3^2)
\\
& = (1-X_1r_1)^2+X_1^2(1-r_1^2)-(1-r_1^2)+(1-r_1^2)(X_2^2+X_3^2)
\\
& = (1-X_1r_1)^2-(1-r_1^2)(1-|X|^2)
\\
& = (1-X_1r_1)^2-(1-r_1^2)(1-r^2) 
\endaligned 
$$
and, therefore,
$$
W(X,Z) = 1 - h(X_1,r), 
\qquad 
h(X_1,r) = {(1-r_1^2)(1-r^2) \over (1-X_1r_1)^2}.
$$

Clearly, for each fixed $r\in(0,r_0]$, the function $h(X_1,r)$ attains its minimum value 
at the point $X_1=-r$, so that 
$$
\aligned 
W(X_1,Z) \leq 1-h(-r,r) 
& = {(1+rr_1)^2-(1-r_1^2)(1-r^2) \over (1+rr_1)^2}
\\
& = {(r+r_1)^2 \over (1+rr_1)^2} =: k(r).
\endaligned 
$$
Since
$$
\aligned 
k'(r)  
& = 2 \, {(r+r_1)(1-r_1^2) \over (1+rr_1)^3}>0, 
\endaligned 
$$
for all $r \geq 0$ we have
$$
\aligned 
W(X,Z)
& \leq k(r_0) = {(r_0+r_1)^2 \over (1+r_0r_1)^2}
\\
& \leq k(1) = 1.
\endaligned 
$$ 
Finally, by choosing $\delta_2 := k(r_0)^{1/2}$ we obtain the desired inequality 
and the proof of the lemma is completed.
\end{proof} 


\subsection*{A symmetrization for non-relativistic fluids} 

For clarity, let us explain our strategy to avoid the zelo velocity problem
in the simpler case of the {\sl non-relativistic} 
Euler equations
\be
\label{Euler11}
\aligned 
&\del_t \rho 
  + \sum_{k=1}^n \del_{x_k} \big( \rho \, u_k \big) = 0,
\\
& \del_t \big( \rho \, u_j \big) 
  + \sum_{k=1}^n \del_{x_k} \bigl( \rho \, u_j u_k + p \, \delta_{jk} \bigr) = 0. 
\endaligned  
\ee
Since this is just \eqref{Euler} 
for the limit case $\eps\to 0$, and since the symmetrization argument of Section~\ref{sym}
is still valid for this case, setting $\eps=0$ in Proposition~\ref{FORM1} 
yields a symmetrization of \eqref{Euler11},
in the form
\begin{align}
\notag
& \del_t z_+ 
   +  (|u|+c(\rho)) \, \ut \cdot \nabla z_+
+ c(\rho) \, |u| \, \tr (E(\ut) \, \nabla \ut) = 0,  
\\
\label{simple-Zm2-11}
&  \del_t z_- 
   +  (|u| - c(\rho)) \, \ut \cdot \nabla z_-  
- c(\rho) \, |u| \, \tr (E(\ut) \, \nabla \ut) = 0,  
\\&\notag
2 \, |u|^2 \big( \del_t \ut + u\cdot \nabla \ut \big) 
      + c(\rho) \, |u| \, E(\ut) \, \nabla z_+ - c(\rho) \, |u| \, E(\ut) \, \nabla z_- = 0.
\end{align}
This symmetrization still has the drawback of having a possibly vanishing coefficient 
(the velocity) in the third equation.

Now, recall that the non-relativistic Euler equations \eqref{Euler11}
are invariant under Galilean transformations and
introduce the coordinate frame translated at some given velocity $U$. 
We denote the new variables and unknowns
with the symbol $\sharp$, that is, 
\be
\label{sharp}
t^\sharp: =t, \qquad x^\sharp:= x-Ut, 
\qquad 
\rho^\sharp:= \rho, \qquad u^\sharp :=u-U.
\ee
We first obtain, by virtue of the Galilean invariance, 
\be
\label{Euler12b}
\aligned 
&\del_t^\sharp \rho^\sharp + 
\sum_{k=1}^n \del_{x_k}^\sharp \big( \rho^\sharp \, u_k^\sharp \big)= 0,
\\
& \del_t^\sharp  \big( \rho^\sharp  \, u^\sharp \big)  
  + \sum_{k=1}^n \del_{x_k}^\sharp
 \bigl( \rho^\sharp \, u_j^\sharp u_k^\sharp + p^\sharp \, \delta_{jk} \bigr) = 0.
\endaligned 
\ee
and then we note that the following 
symmetrization of \eqref{Euler12b} can be deduced from the same  argument as above:
\be
\label{simple-Zm2-12b}
\aligned 
& \del_t^\sharp z_+^\sharp 
   +  (|u^\sharp|+c(\rho^\sharp)) \, \ut^\sharp \cdot \nabla^\sharp z_+^\sharp
+ c(\rho^\sharp) \, |u| \, \tr (E(\ut^\sharp) \, \nabla^\sharp \ut^\sharp) = 0,  
\\
&  \del_t^\sharp z_-^\sharp 
   +  (|u^\sharp| - c(\rho^\sharp)) \, \ut \cdot \nabla z_-^\sharp
- c(\rho^\sharp) \, |u^\sharp| \, \tr (E(\ut^\sharp) \, \nabla^\sharp \ut^\sharp) = 0,  
\\
&
2 \, |u^\sharp|^2 \big( \del_t \ut^\sharp + u^\sharp\cdot \nabla^\sharp \ut^\sharp \big) 
      + c(\rho^\sharp) \, |u^\sharp| \, E(\ut^\sharp) \, \nabla^\sharp z_-^\sharp+  c(\rho^\sharp) \, |u^\sharp| \, E(\ut^\sharp) \, \nabla^\sharp z_-^\sharp = 0.
\endaligned
\ee
Now, the advantage of the symmetrization  \eqref{simple-Zm2-12b}, in comparison with 
 \eqref{simple-Zm2-11}, is obvious: In view of  \eqref{sharp}, 
$u^\sharp$ never vanishes as long  as $u$ remains bounded, 
provided 
the reference velocity $U$ can be chosen so that, say, $2 \, |u| \leq |U|$. 
In turn, Kato's theory ensures the local well-posedness for the system \eqref{simple-Zm2-12b} and,
 as a consequence, for the original system \eqref{simple-Zm2-11}.

In the next section, we will  see that the same strategy works for the relativistic case
provided Galilean transformations are replaced by Lorentz transformations.


\section{Local-well-posedness theory}
\label{exist}

Relying on the symmetric form discovered in Proposition~\ref{FORM1} we are now ready to 
establish the main results of the present paper. We denote here by $H_{ul}^r(\Rn)$ the 
uniformly local Sobolev space of order $r \geq 0$. (Recall that, by definition, the Sobolev norm in these spaces
is computed on unit balls with arbitrary center varying in $\Rn$.)  

\begin{theorem}[Local-in-time solutions in Sobolev spaces] 
\label{theo1}
Consider the relativistic Euler equation for an equation of state $p=p(\rho)$ satisfying the 
hyperbolicity condition \eqref{pressure} together with the following condition near the vacuum  
\be
\label{growth}
\limsup_{\rho \to 0 \atop \rho>0} {c(\rho) \over w(\rho)} < \infty. 
\ee
For every constant $M>0$, there exists $\kappa \in (0,1)$ such that the following property holds.  
Given at $t=0$ an initial data $\rhob, \ub$ belonging to the Sobolev space $H_{ul}^r$ with $r >1 + n/2$ 
and satisfying the constraints
$$
0 \leq \rhob \leq M, \qquad \eps^2 \, |\ub|^2 \leq \kappa, 
$$
there exists a unique local solution $\rho, u$ to the corresponding initial-value problem, 
which is defined up to a some maximal time $T>0$ and satisfies 
$$
\rho, u \in C([0,T), H_{ul}^r(\Rn)) \cap C^1([0,T), H_{ul}^{r-1}(\Rn))
$$
and 
$$
\rho \geq 0, \qquad \eps^2 |u|^2 < 1.
$$
\end{theorem}

The rest of this section is devoted to the proof of Theorem~\ref{theo1}.  

\

\noindent{\it Step 1.}  
By assumption, the initial velocity scalar is bounded away from the light speed and, in consequence, 
thanks to Lemma~\ref{lowerbound}, we can find a vector $U \in \Rn$ with sufficiently large norm $|U|$ such that 
the transformed fluid velocity $\ub'$, defined as in \eqref{addition} by 
\be
\label{additionb}
\aligned 
\ub' 
& = \frac{1}{1-\eps^2 U\cdot \ub}
        \Big(\gamma(U)^{-1}\ub + (1 - \gamma(U)^{-1}) (\Ut\otimes \Ut) \ub - U \Big), 
\\
& = {1 \over \eps} \Phi(\eps \ub, \eps U),
\endaligned
\ee
is bounded and bounded away from the origin. Precisely, for some constants 
$0 < \deltab_1 < \deltab_2 < 1$ we have  
\be
\label{bdvelo}
\deltab_1 \leq \eps |\ub'| \leq \deltab_2. 
\ee

\

\noindent{\it Step 2.} Recall that 
the general theory established by Kato \cite{Kato} covers symmetric hyperbolic systems of the form 
\be
\label{SHS}
A_0(W) \, \del_t W + \sum_{j=1}^n A_j(W) \, \del_j W  = 0, 
\ee
where the $(d\times d)$-matrix fields $A_0, A_j$ are real-valued and symmetric 
with regular coefficients, and the matrix $A_0$ is uniformly positive definite. 
For the system \eqref{simple-Zp2}--\eqref{simple-ut3}, we have $d=n+2$, 
$W=(z_+,z_-,\ut )^t$ (a column vector, the subscript ``$t$"  standing for transposition), and
\begin{align}\label{matrix}
A_0(W)=&\begin{pmatrix}
a_0&0&0
\\
0&b_0&0
\\
0&0&c_0|u|^2 \, I
\end{pmatrix},\quad
A_j(W)=
\begin{pmatrix}
a_1\ut_j&0&a_2|u|e_j
\\
0&b_2\ut_j&- a_2|u|e_j
\\
a_2|u|e_j^t&-a_2|u|e_j^t&c_0|u|^2 u_j \, I
\end{pmatrix},
\end{align}
where
\be
\label{coeff}
\aligned 
& a_0 =1 + \eps^2 |u| c(\rho), \qquad  b_0 = 1 - \eps^2 |u| c(\rho),
 \qquad c_0 = {2 \over 1 - \eps^2 |u|^2},  
\\
& a_1 = {1 - \eps^2 \, c(\rho)^2 \over 1 - \eps^2 c(\rho) |u|} \, (|u| + c(\rho)), 
\qquad 
b_1 = {1 - \eps^2 \, c(\rho)^2 \over 1 + \eps^2 c(\rho) |u|} \, (|u| - c(\rho)),
\\
& a_2 = c(\rho), \qquad e_j = (E_{j1}(u), E_{j2}(u), \ldots E_{jn}(u)),
\endaligned
\ee 
and we recall that $I = (\delta_{ij})$ denotes the $n$-dimensional identity matrix. 
Recall that we are interested in the initial-value problem associated with \eqref{SHS} 
where initial data are prescribed on the initial hyperplane
$$
\Hcal_0 : \quad  t=0.
$$

First,  observe that 
\be
\label{A0}
       \big\langle A_0(W) \, \xi, \xi \big\rangle = a_0 \, |\xi_1|^2 + b_0 \, |\xi_2|^2 + c_0 |u|^2 \, |\xih|^2,
\ee
where $\langle\cdot, \cdot \rangle$ denotes the Euclidian inner product in $\R^{n+2}$ and
$$
\xi = (\xi_1,\xi_2,\ldots, \xi_{n+2}) = (\xi_1, \xi_2, \xih )
\in \R^{n+2},  \qquad \xih=(\xi_3,\ldots,\xi_{n+2})\in \R^n.
$$
As was already noted before Remark~\ref{rm2.3} and is readily seen directly 
from \eqref{coeff}, 
the matrix 
 $A_0$ can be positive definite only if the velocity $u$ never vanishes.  
In other words, provided the initial velocity $\ub$ is bounded away from $0$, 
according to Kato's theory,
a local-in-time solution exists and is unique 
in the uniformly local Sobolev space $H^s_{ul}$ for $s > 1+n/2$. 
As stated in Section 1, however, this lower bound on the velocity is not physically realistic.  

On the other hand, from the physical view point,  
the zero velocity is not a special value and, in any case, 
one should be able to recover the strict positivity property for the matrix $A_0$. 
At this juncture,
recalling the strategy presented at the end of the preceding section for the 
non-relativistic Euler equations,
 we propose ourselves to apply a Lorentz transformation. 

Using the Lorentz invariance property (Lemma~\ref{additionlaw}) of the Euler equations, 
we see that the symmetric formulation \eqref{simple-Zp2}--\eqref{simple-ut3} 
can be also expressed in the transformed coordinates $(t',x')$ defined by \eqref{Lorentz}, that is,
\be
\label{SHSprime}
A_0(W') \, \del_{t'} W' + \sum_{j=1}^n A_j(W') \, \del_{x_j'} W'  = 0, 
\ee
where  $W'=(z'_+, z'_-, \ut') $ is defined from the transformed unknowns $(\rho', u')$. 
(Of course, the mass density remains unchanged but is now regarded as a function of $(t',x')$.)

After this transformation, the expression \eqref{A0} becomes 
\be
\label{A0prime}
\big\langle A_0(W') \, \xi, \xi \big\rangle = a_0' \, |\xi_1|^2 + b_0' \, |\xi_2|^2 + c_0' |u'|^2 \, |\xih|^2,
\ee
with $a_0',b_0',c_0'$ defined by \eqref{coeff} with $(\rho,u)$ replaced by $(\rho',u')$. 
In view of the lower and upper bounds \eqref{bdvelo}, we conclude that the transformed matrix $A_0(W')$ 
is {\sl positive definite in the coordinate system $(t',x')$.} 
Hence, Kato's theory applies to the initial value problem for \eqref{SHSprime}, 
{\sl without} any assumption on the fluid velocity, but provided initial data are imposed on the 
initial hypersurface $t'=0$. 

In contrast to the non-relativisitic case, however, this is not the end of our discussion, 
since $t'=0$ is not the hypersurface of interest. This 
is due to the fact that, in the relativistic setting, 
the initial plane $\Hcal_0$ is not preserved by the transformation \eqref{Lorentz}. 

\

\noindent{\it Step 3.} In fact, the initial hyperplane $\Hcal_0$ is mapped, in the new coordinate system $(t',x')$,
 to the ``oblique'' hyperplane 
$$
\Hcal_0' : \qquad t' = -\eps^2 \, U \cdot x'.
$$
In order to prove local well-posedness for the oblique initial-value problem to \eqref{SHSprime}
with data prescribed on $\Hcal_0'$, it is convenient to introduce a further change of coordinates 
\be
\label{newvar}
t'' = t' + \eps^2 \, U \cdot x',         \qquad x'' = x',
\ee
which maps the hyperplane $\Hcal_0'$ to the hyperplane 
$$
\Hcal_0^{''} : \qquad t'' = 0. 
$$
This transformation puts the system \eqref{SHSprime} into the form 
\be
\label{SHSB}
B_0(W'') \, \del_{t^{''}} W'' + \sum_{j=1}^n B_j(W'')\, \del_{x_j^{''}} W''  = 0, 
\ee
where $W''(t'',x'')=W'(t',x')$ and  the new matrix-coefficients
 are
\begin{equation}\label{doubleprime}
\aligned
& B_0(W'') = A_0(W') + \eps^2 \, \sum_{j=1}^n U_j \, A_j(W'), 
\\
& B_j(W'') = A_j(W'), \qquad j=1,\ldots,n.
\endaligned
\end{equation}
(Note in passing that using a Lorentz transformation instead of \eqref{newvar} would be 
physically more natural, but would lead to precisely the same matrix $B_0$ and slightly 
more complicated expression. 

We are now going to establish that the matrix  $B_0(W'')$ is positive definite
for data that have bounded mass density and whose velocity scalar is bounded away from the light speed.  
Provided this is checked, 
Kato's theory then applies to the initial value problem for \eqref{SHSB} on the hyperplane $\Hcal''_0$, 
which is the one of interest.

At this juncture, it is worth recalling the standard fact that the definite-positivity property 
above implies that the oblique hyperplane $\Hcal'_0$ is a non-characteristic hypersurface 
for the hyperbolic system \eqref{SHSprime} which also is sufficient to imply local well-posedness. 
Namely, the matrix $B_0(W'')$ can be written as
$$
B_0(W'') =\sum_{a=0}^n \nu_a A_a(W''),
$$
where the vector 
$\nu=(\nu_0,\ldots,\nu_n)=(1,\eps^2U)\in \R^{n+1}$
is normal to $\Hcal'_0$. Thus, the positivity property of $B_0(W'')$ implies also 
that $\text{det} (B_0(W'')) \not= 0$.

Let us summarize the expressions we need here. An easy computation with \eqref{doubleprime}
shows that 
\be
\aligned
\label{B}
\big\langle B_0(W'') \, \xi, \xi \big\rangle 
= & \big(a_0'+a_1'\eps^2(U\cdot \ut')\big)\xi_1^2 + 2 a _2' \, \eps^2 \, |u'| \big(U\cdot(E(\ut')\xih)\big) \, \xi_1 
\\
& + \big(b_0' + b_1' \eps^2 \, (U\cdot \ut')\big) \, \xi_2^2 
  -2 a_2'\eps^2 \, |u'| \, \big(U\cdot(E(\ut')\xih)\big) \, \xi_2   
\\
& + c_0' |u'|^2\big(1+\eps^2 (U\cdot u')\big) \,  |\xih|^2,   
\endaligned
\ee
where we recall that $\xi = (\xi_1,\xi_2,\ldots, \xi_{n+2})=(\xi_1, \xi_2, \xih ) 
\in \R^{n+2}$. 
The primed quantities $a_0'$ etc. are still defined by \eqref{coeff} but with $u$ 
replaced by $u'$ determined by the Lorentz transformation associated with the reference velocity $U$. 
 
Hence, we can regard the expression 
$$
Q(\xi; \eps u, \eps c(\rho'), \eps U) : = \big\langle B_0(W'') \, \xi, \xi \big\rangle 
$$
as a polynomial in $\xi$ and a nonlinear function in $\eps u, \eps c(\rho'), \eps U$. 
As before, to simplify the notation, we introduce  
$$
X := \eps \, u, \quad 
Y := \eps c(\rho'), 
\qquad 
Z := \eps \, U, 
$$
to deduce 
$$
\aligned
Q(\xi; X,Y,Z) 
= & \Big(a_0' + a_1' \Phit(X,Z) \cdot Z  \Big) \xi_1^2 
   + 2 a_2' \, |\Phi(X,Z)| \Big(Z \cdot (E(\Phi(X,Z)) \xih)\Big) \, \xi_1 
\\
& + \Big( b_0' + b_1' \Phit(X,Z) \cdot Z \Big) \, \xi_2^2 
  -2 a_2' |\Phi(X,Z)| \, \big(U\cdot(E(\Phit(X,Z))\xih)\big) \, \xi_2   
\\
& + c_0' |\Phi(X,Z)|^2\big(1 + (Z \cdot \Phi(X,Z))\big) \,  |\xih|^2,   
\endaligned
$$
where (after appropriate rescaling in $\eps$ for the coefficients $a_1', b_1', a_2'$)  
\be
\label{coeffprime}
\aligned 
& a_0' = 1 + |\Phi(X,Z)| Y, \qquad  b_0' = 1 - |\Phi(X,Z)| Y,  \qquad  c_0' = {2 \over 1 - |\Phi(X,Z)|^2},  
\\
& a_1' = {1 - Y^2 \over 1 - Y |\Phi(X,Z)|} \, (|\Phi(X,Z)| + Y), 
\qquad 
b_1' = {1 - Y^2 \over 1 + Y |\Phi(X,Z)|} \, (|\Phi(X,Z)| - Y),
\\
& a_2' = Y, 
\endaligned
\ee 
with $\Phit(X,Z) := \Phi(X,Z)/|\Phi(X,Z)|$.
Replacing the coefficients $a_0', a_1'$, etc by their values, we finally obtain 
\be
\label{Q}
Q(\xi; X,Y,Z) 
 = Q_{1} \xi_1^2 + Q_{2} \, \xi_2^2 + Q_{3} \, |\xih|^2 + Q_{13}( \xi_1, \xih) + Q_{23}( \xi_2, \xih), 
\ee
where 
$$
\aligned 
& Q_{1} := 1 + |\Phi(X,Z)| Y + \big( \Phit(X,Z) \cdot Z\big)  {1 - Y^2 \over 1 - Y |\Phi(X,Z)|} \, (|\Phi(X,Z)| + Y), 
\\
& Q_{2} :=  1 - |\Phi(X,Z)| Y + \big( \Phit(X,Z) \cdot Z \big) {1 - Y^2 \over 1 + Y |\Phi(X,Z)|} \, (|\Phi(X,Z)| - Y), 
\\ 
& Q_{3} :=  {2 \, |\Phi(X,Z)|^2 \over 1 - |\Phi(X,Z)|^2} \Big(1 + (Z \cdot \Phi(X,Z))\Big), 
\\
& Q_{13}( \xi_1, \xih) := 2 Y \, |\Phi(X,Z)| \Big(Z \cdot (E(\Phit(X,Z)) \xih)\Big) \, \xi_1 
\\
& Q_{23}( \xi_2, \xih) := -2 Y \, |\Phi(X,Z)| \, \big(Z \cdot(E(\Phit(X,Z))\xih)\big) \, \xi_2, 
\endaligned
$$

Recall that $E(\Phit(X,Z)) = I - \Phit(X,Z) \otimes \Phit(X,Z)$. 
We are interested in the range where $X, Z \in \R^n$ have norm bounded away from $1$, 
and $Y$ remains in a bounded closed subset of $[0,1)$.

It remains to check the following {\sl purely algebraic} result. 

\begin{lemma}[Uniform positivity property]
\label{455} 
For any given $Y_0\in(0,1)$, there exist $r_*\in (0,1)$ and $Z\in\R^n, r_*<|Z|<1$
such that 
\be 
\label{posB_0}
Q(\xi, \xi; X,Y,Z) 
\geq c_0 \, |\xi|^2, \qquad \xi\in\R^{n+2} 
\ee
holds for all $Y \in [0,Y_0] $,  $|X|\leq r_*$.
\end{lemma}

\begin{proof} 
In the below, we fix $Y_0\in(0,1)$  and  $Y\in[0,Y_0]$ is an arbitrary number. 
It is convenient to choose the coordinate system so that
$$
r_0=|X|, \qquad
Z = (r_1,0,0), \qquad 0<r_0<r_1<1.
$$
Set
\[
\Phi_*=\frac{r_1-r_0}{1-r_0r_1}, \qquad
\Phi^*=\frac{r_0+r_1}{1+r_0r_1}.
\]
Then 
\[
0<\Phi_*<\Phi^*<1
\]
and Lemma \ref{lowerbound} says that
\[
\Phi_*
\le 
\big|\Phi(X,Z)\big|\le \Phi^*.
\]
A simple computation shows that
\begin{align*}
Z\cdot \Phi(X,Z) =\frac{\tilde{Z}\cdot X-|Z|}{1-X\cdot Z}|Z|=
\frac{X_1-r_1}{1-X_1r_1}r_1
 =:R <0,
\end{align*}
which, together with $W(X,Z)$ in \eqref{W},  leads to
\begin{align*}
S:&=Z \cdot \Phit(X,Z) =\frac{R}{W(X,Z)^{1/2}}
\\&=|Z|\frac{X_1-r_1}{\big( (X_1-r_1)^2+
\gamma(U)^{-2}(X_2^2+X_3^2)^2\big)^{1/2}}
\end{align*}
Thus, we obtain 
\[
      -r_1=-|Z|\le S\le 0, 
\]
and  the equality in the first inequality is realized for $X=(X_1,0,0)$.

As a consequence, we have,
\begin{align}\notag
Q_{1} &= \displaystyle\frac{
1 - |\Phi(X,Z)|^2 Y^2 
+S  (1 - Y^2)  
  \, (|\Phi(X,Z)| + Y)}{1 - Y |\Phi(X,Z)|}
\\&=\displaystyle\frac{
(1 - |\Phi(X,Z)|^2)Y^2 \notag
+(1-Y^2)\left(1+S(|\Phi(X,Z)| + Y)\right)}{1 - Y |\Phi(X,Z)|}
\\
\label{Q11lowerb}
&\ge
(1-Y_0^2)\, \big( 1 - r_1( \Phi^* + Y_0) \big) =:q_1.
\end{align}
Therefore, we see that 
 if 
\begin{equation}\label{cond1}
k_0:=r_1(\Phi^*+Y_0)< 1,
\end{equation}
then $q_1>0$ and so, $Q_{1}>0$. 

Similarly,
\begin{align*}
Q_{2} &= \displaystyle\frac{
1 - |\Phi(X,Z)|^2 Y^2 
+ S  (1 - Y^2)  
  \, (|\Phi(X,Z)| - Y)}{1 + Y |\Phi(X,Z)|}
\\&=\displaystyle\frac{
(1 - |\Phi(X,Z)|^2)Y^2 
+(1-Y^2) \, \left(1+ S(|\Phi(X,Z)| - Y)\right) }{1 + Y |\Phi(X,Z)|}
\\&
\ge
\frac{
(1-Y_0^2)\Big(1-r_1\max(0, |\Phi(X,Z)| - Y)\Big) }{1+Y_0\Phi^*}.
\end{align*}
Therefore, we have 
$$
Q_{2} 
\ge
\frac{
(1-Y_0^2)(1 - r_1\Phi^*) }{1+Y_0\Phi^*}=:q_2>0, 
$$
since $|\Phi(X,Z)| - Y\le |\Phi(X,Z)| \le \Phi^*$.

On the other hand, clearly we have 
\[
Q_{3}\ge {2 \, \Phi_*^2 \over 1 - \Phi_*^2} \Big(1 -\Phi^*r_1\Big)=:q_3>0, 
\]
and 
\[
|Q_{13}( \xi_1, \xih)|\le 2 Y \, |\Phi(X,Z)| |Z||\xi_1||\xih|
\le 2Y_0\Phi^*r_1|\xi_1||\xih|,
\]
\[
|Q_{23}( \xi_1, \xih)|\le 2 Y \, |\Phi(X,Z)| |Z||\xi_2||\xih|
\le 2Y_0\Phi^*r_1|\xi_1||\xih|.
\]
Set now 
\[
q_4:=Y_0\Phi^*r_1
\]
and define
a quadratic formula of three variables $(x,y,z)\in\R^3$,
\[
Q_*(x,y,z)
=q_1x^2+q_2y^2+q_3z^2-
2q_4(x+y)z.
\]
Then, we obtain 
\[
Q(\xi; X,Y,Z)\ge Q_*(|\xi_1|, |\xi_2|, |\xih|)
\]
for any $\xi\in\R^{n+2}$; 
that is,  $Q$ is  positive definite if so is $Q_*$.

Since $q_1, q_2>0$, we can write, 
 for any $\kappa\in(0,1)$,
\begin{align*}
Q_*(x,& y,z)
 =\kappa(q_1 x^2+q_2 y^2)+(1-\kappa)q_1\Big(x-\frac{q_4}{(1-\kappa)q_1}z\Big)^2
\\&
+(1-\kappa)q_2\Big(y-\frac{q_4}{(1-\kappa)q_2}z\Big)^2+
\Big(q_3-\frac{q_4^2}{(1-\kappa)q_1}-\frac{q_4^2}{(1-\kappa)q_2}\Big)z^2,
\end{align*}
and since $\kappa\in(0,1)$ is arbitrary,  we can conclude that 
$Q_*$ is positive definite if and only if
\be
\label{cond2}
D_*:=q_3-\frac{q_4^2}{q_1}-\frac{q_4^2}{q_2}>0.
\ee

Let $r_*\in (0,1)$ be a number to be determined later. Let $a\in (1, 1/r_*)$ and set
 $r_1=ar_*$. Then, observe that for any $r_0\in[0,r_*]$,
 \[
 \Phi_*=\frac{r_1-r_0}{1-r_0r_1}\ge \frac{(a-1)r_*}{1-ar_*^2}, \qquad
\Phi^*=\frac{r_0+r_1}{1+r_0r_1}\le\frac{(a+1)r_*}{1+r_0r_1}\le (a+1)r_*.
\]
Hence,
\begin{align*}
k_0&=r_1(\Phi^*+Y_0)\le ar_*((a+1)r_*+Y_0)\le a(a+2)r_*=:k_1r_*,
\\
q_1&=(1-Y_0^2)(1-k_0)\ge (1-Y_0^2)(1-a(a+2)r_*)=: K_1(r_*),
\end{align*}
and 
\begin{align*} 
q_2&=\frac{(1-Y_0^2)(1 - r_1\Phi^*)}{1 + Y_0r_1}\ge 
\frac 12(1-Y_0^2)(1- a(a+1)r_*^2)=:K_2(r_*),
\\
q_3&={2 \, \Phi_*^2 \over 1 - \Phi_*^2} \Big(1 -\Phi^*r_1\Big)\ge
\frac{2(a-1)^2r_*^2 \, (1-a(a+1)r_*^2) }{(1-ar_*^2)^2-(a-1)^2r_*^2}=: K_3(r_*)r_*^2,
\\
q_4&=Y_0\Phi^*r_1\le Y_0(a+1)ar_*^2=:K_4r_*^2. 
\end{align*}
Note that $k_1, \ K_4>0$ are independent of $r_*$ and that all the above inequalities
hold for all $r_0\in[0,r_*]$.
Observe that if $r_*$ is sufficiently small,
then $K_1(r_*)$, $K_2(r_*)$, $K_3(r_*)$ are positive and  
\begin{equation}\label{cond3}
D_*\ge \Big(
K_3(r_*)-K_4\Big(
\frac{1}{K_1(r_*)}+\frac{1}{K_2(r_*)}\Big)
r_*^2\Big) \, r_*^2
\end{equation}
holds for all $r_0$.
It is easy to see that when $r_*\to 0$, we have 
$$
k_0\to 0, \quad
K_1(r_*)\to 1-Y_0^2, 
$$
and 
$$
K_2(r_*)\to \frac 12(1-Y_0^2), \quad K_3(r_*)\to 2(a-1)^2. 
$$
We can now conclude from \eqref{cond3} that 
{\sl for all  $a>1$, there exists $r_* \in (0,1)$ such that 
$ar_*<1$,  $k_0 <1$,  $D_{*}>0$,}
 which 
completes the proof of Lemma~\ref{455} with $r_1=|Z|=ar_*$.
\end{proof}
 

In turn, Kato's theory guarantees the existence of a solution 
defined in a small neighborhood 
of this hyperplane $\Hcal''_0$.
Making the transformation back to the original variables, we obtain 
a solution in a small neighborhood
of the initial line $t=0$.  This completes the proof of  Theorem \ref{theo1}.

\

\noindent{\it Step 4.} We need to show that the density $\rho$ remains non-negative. In fact, 
since the solution remains smooth enough, say of class $C^1$ in space and continuous in time, 
we can define the characteristic curves by the Cauchy-Lipschitz theorem 
$$
\dot{y}(t) = u(t, y(t)), \qquad t \geq 0. 
$$
Hence, by writing the $w$-equation as 
$$
\aligned 
& \big( \del_t w + u \cdot \nabla w \big) 
\\
& = - {c(\rho) \over 1 - \eps^4 |u|^2 c(\rho)^2} \big( \big(1 - \eps^2 |u|^2\big) \, \big(  \eps^2 \, c(\rho) \, u \cdot \nabla w
     + \ut \cdot \nabla v  \big) + |u| \nabla \cdot \ut \big),  
\endaligned 
$$
and integrating along the characteristics we obtain 
$$
{d \over dt}\big( w(t, y(t)\big) = O(1) \, c(\rho(t, y(t)). 
$$
Here, the (bounded) quantity $O(1)$ involves only the sup-norm of first-order derivatives 
of the solution. 

Now, in view of our assumption \eqref{growth} we have 
$$
c(\rho) \leq C_R \, w(\rho) 
$$
on any compact set $\rho \in [0,R]$ and for some constant $C_R$. So, we deduce that 
$$
w(t,y(t)) = w(0, y(0)) \, e^{t \, O(1)}, 
$$
along every characteristic, which, in particular, shows that $w$ remains non-negative. 
This completes the proof of Theorem~\ref{theo1}.


\section{Support of tame solutions}  
\label{general}  

Following \cite{MakinoUkaiKawashima} we define a concept of solutions, in which the velocity vector 
is required to satisfy a (non-degenerate) evolution equation even in the presence of vacuum, as follows.

\begin{definition}[Notion of tame solution]
 A measurable map $(\rho,u) : [0,T] \times [0,\infty) \times [-\eps^{-2}, \eps^{2}]$ 
is called a {\rm tame solution} of the relativistic Euler equations if 
\begin{itemize}

\item $(\rho,u)$ is a solution of class $C^1$ of the Euler equations, 

\item $w$ is also of class $C^1$, and 

\item the equation $\del_t u + u\cdot \nabla u =0$ holds in the interior of the set 
$\big\{ \rho =0 \big \}$. 
\end{itemize}

\end{definition}

\

Relying on this definition, we can establish that the support of a solution does not expand in time. 

\begin{theorem}[Property of the support of a tame solution]
 Consider the relativistic Euler equation for an equation of state $p=p(\rho)$ satisfying 
the hyperbolicity condition \eqref{pressure} together with the following vacuum condition 
\eqref{growth}. 
If $(\rho,u)$ is a tame solution of the relativistic Euler equations
and has compact support,  
then its support does not expand in time, that is, 
$$
\supp(\rho,u)(t) \subset \supp(\rhob, \ub), \qquad t \in [0,T]. 
$$
\end{theorem}

\begin{proof} We now consider the relativistic Euler equations in the form 
$$
B_0(V) \, \del_t V + \sum_{j=1}^n B_j(V) \, \del_j V  = 0, 
$$
where the map $V := (w,v, \ut)$ is of class $C^1$ and satisfies
$$
|\del_t V | \leq C_1 \, \sup_j |B_0(V)^{-1} B_j(V)|.   
$$
Here, the constant $C_1$ depends on the sup norm of first-derivatives of the solution. 
In view of the explicit expressions of the matrices $B_0$ and $B_j$ (see Section~\ref{sym})
we obtain 
$$
\sup_j |B_0(V)^{-1} B_j(V)| \leq C_2 \, |V|, 
$$
where we have used our assumption \eqref{growth}. 
Now, by Gronwall's lemma we get
$$
|V(t,x)| \leq e^{Ct}|V(0,x)|, \qquad t>0, \ x\in\R^n,
$$
which completes the proof of the theorem.
\end{proof}


\section*{Acknowledgments}

The first author (PLF) was partially supported by the Centre
National de la Recherche Scientifique (CNRS) and the Agence
Nationale de la Recherche (ANR) through the grant 06-2-134423 (MATH-GR). 
The research of the second author (SU) was supported by the Department of Mathematics 
and Liu Bie Ju Centre for Mathematical Sciences at the City University of Hong Kong. 
This work was initiated during a two-week visit of the first author in these institutions, and 
both authors are very grateful to Tong Yang for inviting them and particularly appreciated his hospitality. 

This paper was completed when the first author was visiting the Mittag-Leffler Institute
in the Fall 2008 during the Semester Program ``Geometry, Analysis, and General Relativity''
organized by L. Andersson, P. Chrusciel, H. Ringstr\"om, and R. Schoen.


\end{document}